\documentclass[10pt]{amsart}

\usepackage{enumerate}
\usepackage{amsmath,amssymb, amsthm}

\newtheorem{theorem}{Theorem}[section]
\newtheorem{corollary}[theorem]{Corollary}
\newtheorem{lemma}[theorem]{Lemma}
\newtheorem{problem}[theorem]{Problem}
\newtheorem{proposition}[theorem]{Proposition}

\theoremstyle{definition}
\newtheorem{remark}[theorem]{Remark}

\def\A{{\mathcal{A}}}
\def\B{{\mathcal{B}}}
\def\E{{\mathcal{E}}}
\def\F{{\mathcal{F}}}
\def\H{{\mathcal{H}}}
\def\P{{\mathcal{P}}}
\def\U{{\mathcal{U}}}

\begin{document}

\title[Tingley's problem for operator algebras]{Tingley's problem through\\
the facial structure of operator algebras}

\author[M. Mori]{Michiya Mori}

\address{Graduate School of Mathematical Sciences, the University of Tokyo, Komaba, Tokyo, 153-8914, Japan.}
\email{mmori@ms.u-tokyo.ac.jp}



\date{}

\begin{abstract}
Tingley's problem asks whether every surjective isometry between the unit spheres of two Banach spaces admits an extension to a real linear surjective isometry between the whole spaces.
In this paper, we give an affirmative answer to Tingley's problem when both spaces are preduals of von Neumann algebras, the spaces of self-adjoint operators in von Neumann algebras or the spaces of self-adjoint normal functionals on von Neumann algebras.  
We also show that every surjective isometry between the unit spheres of unital C$^*$-algebras restricts to a bijection between their unitary groups. 
In addition, we show that every surjective isometry between the normal state spaces or the normal quasi-state spaces of two von Neumann algebras extends to a linear surjective isometry.
\end{abstract}
\maketitle
\thispagestyle{empty}

\section{Introduction}
In 1951, the study of isometries between operator algebras began in \cite{Ka51}. 
In that paper, Kadison proved that every complex linear surjective isometry between two unital C$^*$-algebras can be decomposed as the composition of a Jordan $^*$-isomorphism and the multiplication by a unitary. 
Since then, linear isometries between operator algebras have been considered in various settings by many researchers. 
For example, see \cite{FMP} and \cite{S}, which contain results and references concerning generalizations of Kadison's theorem to thoroughly different directions.\smallskip

On the other hand, the Mazur-Ulam theorem states that every surjective isometry between two real normed spaces is affine. 
This result attracted many mathematicians, and isometries without assuming affinity were considered in many cases.
The symbol $S(X)$ denotes the unit sphere (i.e.\ the subset of the elements with norm $1$) of a Banach space $X$, while the notation $\B_X$ means the closed unit ball of $X$. 
What we focus on in this paper is the following problem, which is closely related to the Mazur-Ulam theorem. 

\begin{problem}\label{Tingley}
Let $X$ and $Y$ be two Banach spaces and $T: S(X)\to S(Y)$ be a surjective isometry. 
Does $T$ admit an extension to a real linear surjective isometry $\widetilde{T}: X\to Y$?
\end{problem}

The first contribution to this problem dates back to 1987, and it is due to Tingley \cite{Ti}, so this problem is nowadays called \emph{Tingley's problem} (or the \emph{surjective isometric extension problem}).
More than 30 years have passed since the birth of this problem, but the answer in general situations is yet far from having been achieved. 
Indeed, it is said that Tingley's problem is unsolved even in the case $X = Y$ and $X$ is two dimensional. 
However, until now, no counterexamples for Tingley's problem have been found. 
Moreover, in many cases (including the cases of most of classical real Banach spaces and some special Banach spaces), affirmative answers have been given for Tingley's problem. 
The survey \cite{Di} contains good expositions and references on Tingley's problem.\smallskip

Tingley's problem in the setting of operator algebras was first considered by Tanaka \cite{Ta16}, and he later solved Tingley's problem affirmatively when $X$ and $Y$ are finite von Neumann algebras \cite{Ta17}. 
Recently, Fern\'andez-Polo and Peralta generalized this result to the cases of general von Neumann algebras \cite{FP}. 
On the other side, Fern\'andez-Polo, Garc\'es, Peralta and Villanueva solved Tingley's problem positively when $X$ and $Y$ are the spaces of trace class operators on complex Hilbert spaces \cite{FGPV}. 
See Introduction of \cite{FP} for the latest developments in this field.
It is common to use the following strategy to solve Tingley's problem for operator algebras. 
First we detect some substructures of the unit spheres such as unitary groups and minimal or maximal partial isometries. 
In this step, the facial structure of unit balls plays a crucial role. 
Second we construct the only one candidate for the real linear extension which is determined by such substructures. 
And finally we show that this linear mapping is the extension we wanted.\smallskip

In this paper, applying some versions of this strategy, we give several new results concerning Tingley's problem in the setting of operator algebras.

In Section \ref{face}, we summarize some known results about the facial structure of operator algebras and (pre)duals (due to Akemann and Pedersen \cite{AP}) and its application to Tingley's problem, which are used in the later sections.

In Section \ref{C*}, we show that every surjective isometry between the unit spheres of two unital C$^*$-algebras restricts to a bijection between their unitary groups. In the proof, we detect the unitary group from extreme points in the unit ball. Using the surjective isometry between unitary groups and the result due to Hatori and Moln\'ar \cite{HM}, we construct the only one candidate for the real linear isometric extension. Although the author does not know whether this linear mapping actually extends the original mapping, we show that Tingley's problem for unital C$^*$-algebras is equivalent to Problem \ref{problem-C*}.

In Section \ref{predual}, we give a positive answer to Tingley's problem when $X$ and $Y$ are preduals, $M_*$ and $N_*$ of von Neumann algebras $M$ and $N$, respectively. In the proof, we use the structure of maximal faces, and calculate Hausdorff distances between them to construct a surjective isometry between the unitary groups of $M$ and $N$. By the theorem of Hatori and Moln\'ar, this mapping extends to a real linear surjective isometry from $M$ onto $N$. This linear mapping canonically determines a real linear surjective isometry from $N_*$ onto $M_*$, whose inverse mapping is shown to be the extension we wanted.

In Section \ref{sa}, we show that Tingley's problem has an affirmative answer when $X$ and $Y$ are the spaces $M_{sa}$ and $N_{sa}$ of self-adjoint operators in von Neumann algebras $M$ and $N$, respectively. In this case, some techniques used in sections before cannot be applied. Instead, we use the structure of projection lattices and orthogonality combined with a theorem of Dye \cite{Dy}.
We also solve Tingley's problem positively when $X$ and $Y$ are the spaces $M_{*sa}$ and $N_{*sa}$ of self-adjoint elements in preduals of von Neumann algebras $M$ and $N$, respectively. 
Additionally, applying some discussions in this paper, we show that every surjective isometry $T: S(M_{*})\cap M_{*+} \to S(N_{*})\cap N_{*+}$ (resp. $T: \B_{M_{*}}\cap M_{*+} \to \B_{N_{*}}\cap N_{*+}$) between the normal state spaces (resp. between the normal quasi-state spaces) of two von Neumann algebras $M$ and $N$ admits a linear surjective isometric extension from $M_{sa}$ onto $N_{sa}$.

In Section \ref{problem}, along the line of this paper, we list problems which seem to be open and new, with some comments.

\section{Facial structure of operator algebras and its use in Tingley's problem}\label{face}
Recall that a nonempty convex subset $\F$ of a convex set $\mathcal{C}$ in a Banach space is called a \emph{face} in $\mathcal{C}$ if $\F$ has the following property: if $x, y \in\mathcal{C}$ and $\lambda x + (1-\lambda) y \in \F$ for some $0 < \lambda < 1$, then $x, y \in \F$.
It can be easily proved by Hahn-Banach theorem that for a Banach space $X$, a subset $\F$ of $\B_X$ is a maximal norm-closed proper face in $\B_X$ if and only if $\F$ is a maximal convex subset of $S(X)$ (see \cite[Lemma 3.2]{Ta17}).
In order to attack Tingley's problem, nowadays the following geometric property is known: 
every surjective isometry between the unit spheres of two Banach spaces preserves maximal convex sets of the spheres (\cite[Lemma 5.1$(ii)$]{CD}, \cite[Lemma 3.5]{Ta14}). \smallskip

On the other hand, the facial structure of the unit ball of operator algebras and (pre)duals were thoroughly studied by Akemann and Pedersen \cite{AP}.
Let $X$ be a real or complex Banach space and $F \subset X$, $G \subset X^*$ be subsets. We define 
\[
\begin{split} 
F^{\prime} &:= \{f \in \B_{X^*} \mid f(x) = 1\,\, \text{for any}\,\, x \in F\},\\
G_{\prime} &:= \{x \in \B_X \mid f(x) = 1\,\, \text{for any}\,\, f \in G\}.
\end{split}
\]

\begin{theorem}[Akemann and Pedersen {\cite[Theorem 5.3]{AP}}]\label{thm-AP1}
Let $X$ be one of the following Banach spaces: 
a C$^*$-algebra, the space of self-adjoint operators in a C$^*$-algebra, the predual of a von Neumann algebra, or the space of self-adjoint elements in the predual of a von Neumann algebra. 
(Consider $X$ as a complex Banach space in the first or the third case, and real in the other cases.) 
Then the mapping $\F \mapsto \F^{\prime}$ is an order-reversing bijection from the class of norm-closed faces in $\B_X$ onto the class of weak$^*$-closed faces in $\B_{X^{*}}$. 
The inverse mapping is given by $\mathcal{G} \mapsto \mathcal{G}_{\prime}$.
\end{theorem} 
Using this theorem as in the proof of Corollary 3.4 in \cite{Ta17} 
(or by Corollary 2.5 of \cite{FGPV}, which can also be applied in the situations of real Banach spaces), 
we obtain the following proposition. For the convenience of the readers, we add a proof.

\begin{proposition}[A version of {\cite[Corollary 3.4]{Ta17}} or {\cite[Corollary 2.5]{FGPV}}]\label{prop-T}
Let $A$ and $B$ be C$^*$-algebras, $M$ and $N$ be von Neumann algebras and the pair $(X, Y)$ be one of the following pairs: 
$(A, B)$, $(A_{sa}, B_{sa})$, $(M_*, N_*)$ or $(M_{*sa}, N_{*sa})$. 
Suppose $T: S(X) \to S(Y)$ is a surjective isometry. 
Then for a subset $\F \subset S(X)$, $\F$ is a norm-closed proper face in $\B_X$ if and only if $T(\F)$ is in $\B_Y$. 
In particular, $x\in S(X)$ is an extreme point in $\B_X$ if and only if $T(x)$ is in $\B_Y$. 
\end{proposition}
\begin{proof}
Let $\F$ be a norm-closed proper face in $\B_X$. By the preceding theorem and the Krein-Milman theorem, we have 
\[
\F = (\F^{\prime})_{\prime} = \bigcap_{f \in \F^{\prime}} \{ f \}_{\prime} = 
\bigcap_{f \in \operatorname{ext} (\F^{\prime})} \{ f \}_{\prime}.
\]
Since $\F^{\prime}$ is a face, it follows that $\operatorname{ext} (\F^{\prime}) \subset \operatorname{ext} (\B_{X^*})$. 
Again by the preceding theorem, for every $f \in \operatorname{ext} (\B_{X^*})$, $\{ f \}_{\prime}$ is a maximal norm-closed proper face in $S(X)$. 
By the fact that $T$ gives a bijection between the classes of maximal norm-closed proper faces in unit balls, it follows that the set 
\[
T(\F) = T\biggl( \bigcap_{f \in \operatorname{ext} (\F^{\prime})} \{ f \}_{\prime} \biggr) = \bigcap_{f \in \operatorname{ext} (\F^{\prime})} T\left(\{ f \}_{\prime}\right)
\]
is a norm-closed proper face in $\B_{Y}$.
\end{proof}

We add a little more to these results.
\begin{proposition}[See also {\cite[Section 4]{Ti}}]\label{prop-face}
Let $X$ and $Y$ be Banach spaces and suppose that $T: S(X) \to S(Y)$ is a surjective isometry.
\begin{enumerate}[$(a)$]
\item Let $\F \subset S(X)$ be a maximal convex set. Then $T(-\F) = -T(\F)$.
\item Suppose $(X, Y)$ is a pair as in the preceding proposition and let $\F \subset \B_X$ be a norm-closed proper face. Then $T(-\F) = -T(\F)$. 
\end{enumerate}
\end{proposition}
\begin{proof}
$(a)$ It suffices to show that $-\F = \{ x \in S(X) \mid \| x - y \| = 2 \,\,\text{for any}\,\, y \in \F \}$. 
Let $y_1, y_2 \in \F$. Then $(y_1 + y_2)/2 \in \F$. 
In particular, $\|-y_1-y_2\| = \|y_1+y_2\| = 2$. 
Thus we obtain $-\F \subset \{ x \in S(X) \mid \| x - y \| = 2 \,\,\text{for any}\,\, y \in \F \}$. 
Let $x \in S(X)$ and suppose $\| x - y \| = 2$ for all $y \in \F$. 
Then the open convex sets $\mathcal{S}_1 := \{z_1 \in X \mid \operatorname{dist} (z_1, \F) < 1\} = \F + \operatorname{int}\B_X$ and 
$\mathcal{S}_2 := \{z_2 \in X \mid \|z_2-x\| < 1\} = x + \operatorname{int}\B_X$ do not have a common element.
By the Hahn-Banach theorem, we obtain a functional $f \in S(X^*)$ and a real number $c\in \mathbb{R}$ such that 
$\operatorname{Re}f(z_1) > c$ for every $z_1 \in \mathcal{S}_1$ and $\operatorname{Re}f(z_2) < c$ for every $z_2 \in \mathcal{S}_2$. 
Since $\operatorname{Re}f(x),\, \operatorname{Re}f(y)\in [-1, 1]$ for every $y\in \F$ and 
$\operatorname{Re}f(\operatorname{int}\B_X) = (-1, 1)$, 
we have $c = 0$, $\operatorname{Re}f(x) = -1$, $\operatorname{Re}f(y) = 1$ and thus 
 $f(x) = -1$ and $f(y) = 1$. 
It follows that $f^{-1}(1) \cap\B_X \subset \B_X$ is a norm-closed face which contains $\F$. 
By the maximality of $\F$, we have $f^{-1}(1)\cap\B_X = \F$.
Thus $x \in f^{-1}(-1)\cap\B_X = (-f^{-1}(1))\cap\B_X = -\F$.\smallskip

\noindent
$(b)$ follows by $(a)$ and the fact that every norm-closed face is the intersection of some maximal convex sets in $S(X)$ (see the proof of the preceding proposition).
\end{proof}

In fact, Akemann and Pedersen concretely described the facial structure of operator algebras and (pre)duals in order to prove Theorem \ref{thm-AP1}.

Let $A$ be a (not necessarily unital) C$^*$-algebra. 
The partial order in the set of partial isometries in $A$ is given by the following: $u$ majorizes (or extends) $v$ if $u = v + (1-vv^*) u (1-v^*v)$.
A projection $p$ in the bidual $A^{**}$ (considered as the enveloping von Neumann algebra) is said to be \emph{open} 
if there exists an increasing net of positive elements in $A$ converging to $p$ in the $\sigma$-strong topology of $A^{**}$. 
A projection $p \in A^{**}$ is said to be \emph{closed} if $1-p$ is open. 
A closed projection $p$ in $A^{**}$ is \emph{compact} if $p\leq a$ for some norm-one positive element $a \in A$. 
A partial isometry $v\in A^{**}$ \emph{belongs locally to $A$} if 
$v^*v$ is a compact projection and there exists a norm-one element $x$ in $A$ such that $xv^*=vv^*$. 
See \cite{AP} for more information.\smallskip

\begin{theorem}[Akemann and Pedersen \cite{AP}]\label{thm-AP2}
Let $A$ be a C$^*$-algebra and $M$ be a von Neumann algebra.
\begin{enumerate}[$(a)$]
\item For each norm-closed face $\F$ of $\B_A$, there exists a unique partial isometry $v$ belonging locally to $A$ such that 
$\F= \{x\in \B_A \mid xv^* = vv^*\}$. 
\item For each norm-closed face $\F$ of $\B_{A_{sa}}$, there exists a unique pair of compact projections $p, q$ such that $pq = 0$ and $\F = \{x\in \B_{A_{sa}} \mid x(p-q) = p+q\}$.
\item For each weak$^*$-closed proper face $\mathcal{G}$ of $\B_{A^{*}}$, there exists a unique nonzero partial isometry $v$ belonging locally to $A$ such that $\mathcal{G} = \{v\}_{\prime}$.
\item For each weak$^*$-closed proper face $\mathcal{G}$ of $\B_{A^*_{sa}}$, there exists a unique pair of compact projections $p, q$ such that $p+q\neq 0$, $pq = 0$ and $\mathcal{G} = \{p-q\}_{\prime}$.
\item For each $\sigma$-weakly closed face $\mathcal{G}$ of $\B_M$ (resp. $\B_{M_{sa}}$), there exists a unique partial isometry (resp. self-adjoint partial isometry) $v$ in $M$ such that 
\[
\begin{split}
\mathcal{G} &= \{x\in \B_M \mid xv^* = vv^*\} = v + (1 - vv^*)\B_M(1 - v^*v) \\
(\text{resp.}\quad \mathcal{G} &= \{x\in \B_{M_{sa}} \mid xv = v^2\} = v + (1 - v^2) \B_{M_{sa}} (1 - v^2) ).
\end{split}
\]
\item For each norm-closed proper face $\F$ of $\B_{M_{*}}$ (resp. $\B_{M_{*sa}}$), there exists a unique nonzero partial isometry (resp. self-adjoint partial isometry) $v$ in $M$ such that $\F = \{v\}_{\prime}$.
\end{enumerate}
\end{theorem}

See also \cite{ER} for a variant of this result in the setting of JBW$^*$-triples.

\section{On Tingley's problem between unital C$^*$-algebras}\label{C*}
For a unital C$^*$-algebra $A$, the symbol $\U(A)$ will denote the group of unitaries in $A$, and $\P(A)$ stands for the set of projections in $A$. These substructures contain a lot of information about $A$. What we focus on in this section is the group $\U(A)$. 

In the proof of \cite[Theorem 4.12]{Ta17}, Tanaka showed that if $T : S(M)\to S(N)$ is a surjective isometry between the unit spheres of two finite von Neumann algebras, then $T$ restricts to a bijection between their unitary groups, i.e.\ $T(\U(M)) = \U(N)$.
Recently, this result was extended to the case of general von Neumann algebras by Fern\'andez-Polo and Peralta \cite[Theorem 3.2]{FP}.
We further extend these results to the case of arbitrary unital C$^*$-algebras using somewhat a different method. We would like to use the notation $\E(X) := \operatorname{ext} (\B_X)$ for the set of extreme points of $\B_X$ where $X$ is a Banach space. 

Recall that, if $A$ is a unital C$^*$-algebra, then   
\[\E(A) = \{x \in S(A) \mid (1-xx^*)A(1-x^*x) = \{0\}\}
\] 
is the set of maximal partial isometries in $A$ and in particular $\U(A) \subset \E(A)$ (see for example \cite[Theorem 7.3.1]{KR}).

\begin{lemma}\label{lem-C*}
Let $A$ be a unital C$^*$-algebra and $x \in \E(A)$. Then $x$ is in $\U(A)$ if and only if the set
$\A_x := \{ y \in \E(A) \mid \| x \pm y\| = \sqrt{2}\}$
has an isolated point as a metric space.
\end{lemma}

The idea of this lemma comes from the easiest case $A = \mathbb{C}$: 
for $x \in \E(A) = \U(A) = \{ z \in \mathbb{C} \mid |z| = 1\}$, we see $\A_x = \{ix, -ix\}$.
\begin{proof}
First realize $A$ as a unital C$^*$-subalgebra of some $\B(\H)$ (the algebra of bounded linear operators on a complex Hilbert space $\H$).

Suppose $x$ is in $\U(A)$. 
For $y \in \A_x$, we have $2 = \|x \pm y\|^2 = \|1 + y^*y \pm (x^*y + y^*x)\|$.
Decompose $\H$ in the form
$\H =  y^*y \H \oplus (1 - y^*y) \H$.
Using this decomposition, we express
\[
x^*y =
\begin{pmatrix}
z_1 & 0 \\
z_2 & 0
\end{pmatrix}
:
\begin{matrix}
 y^*y \H \\
\oplus \\
(1 - y^*y) \H
\end{matrix}
\to
\begin{matrix}
 y^*y \H \\
\oplus \\
(1 - y^*y) \H
\end{matrix}.
\]
By the same decomposition, we can express
\[
1 + y^*y \pm (x^*y + y^*x) =
\begin{pmatrix}
2 \pm (z_1 + z_1^*) & \pm z_2^* \\
\pm z_2 & 1
\end{pmatrix}.
\]
Since $2 \in 2 \E(\B(y^*y \H))$, by the norm condition we obtain $z_1 + z_1^* = 0$ and $z_2 = 0$. 
Since $x \in \U(A)$, it follows that $x^*y \in \E(A)$. Combining this with the equation
\[
x^*y =
\begin{pmatrix}
z_1 & 0 \\
0 & 0
\end{pmatrix}
 =
\begin{pmatrix}
-z_1^* & 0 \\
0 & 0
\end{pmatrix},
\]
we have $x^*y \in \U(A)$ and the spectrum $\sigma(x^*y)$ of $x^*y$ is a subset of $\{i, -i\}$.
It follows that $\A_x = ix (1-2\P(A)) = i (1-2\P(A))x$, which has isolated points $\pm ix$. \smallskip

Next suppose $x \notin \U(A)$ and $y \in \A_x$. We show $y$ is not isolated in $\A_x$. We may assume $xx^* \neq 1$.
Suppose $(1-xx^*)y \neq 0$. For $c \in \mathbb{T} := \{z \in \mathbb{C} \mid |z| = 1\}$, set $y'_c := (xx^* + c(1-xx^*))y \, (\in \E(A))$. Then we have 
\[
\|x \pm y'_c\| = \|x \pm (xx^* + c(1-xx^*))y\| = \|(xx^* + \overline{c}(1-xx^*))x \pm y\| = \|x \pm y\| = \sqrt{2}.
\] 
Hence $y'_c \in \A_x$. Since $y'_c \to y\, (c \to 1)$, $y$ is not isolated in $\A_x$. 
Similarly, $y$ is not isolated in $\A_x$ if $(1-x^*x)y^* \neq 0$.
In what follows, we assume $(1-xx^*)y = 0 = (1-x^*x)y^*$. Then we obtain $xx^* \geq yy^*$ and $x^*x \geq y^* y$.

Since $y \in \E(A)$, we have $(1-yy^*)A(1-y^*y) = 0$. Taking the closure in the sot of $\B(\H)$, we also have $(1-yy^*)\overline{A}^{sot}(1-y^*y) = 0$. 
By the theory of von Neumann algebras, there exists a central projection $p$ in $\overline{A}^{sot}$ such that $y_1 := yp$ is an isometry on $p\H$ and $y_2 := y(1-p)$ is a coisometry (i.e.\ the adjoint operator of an isometry) on $(1-p)\H$. 
Set $x_1 := xp$ and $x_2 := x(1-p)$. Then it follows that $x_1^*x_1 = y_1^*y_1$, $x_1x_1^* \geq y_1y_1^*$ and $x_2x_2^* = y_2y_2^*$, $x_2^*x_2 \geq y_2^*y_2$. 
If $x_1x_1^* \neq y_1y_1^*$ or $x_2^*x_2 \neq y_2^*y_2$, then we have 
$\sqrt{2} = \| x \pm y \| = \max_{n = 1,2} \|x_n \pm y_n\| = 2$, a contradiction.
It follows that $xx^* = yy^*$ and $x^*x = y^*y$.
The same discussion as in the first half of this proof shows that there exists a projection $q$ in $A$ with $q \leq xx^*$ such that 
$y = i(1-2q)x = i(xx^*-2q)x$.

Suppose that $q$ is isolated in $\P(xx^*Axx^*)$. 
Let $a \in (xx^*Axx^*)_{sa}$. 
Since the mapping $\mathbb{R}\ni t\mapsto e^{ita} q e^{-ita} \in \P(xx^*Axx^*)$ is norm-continuous, we obtain $e^{ia} q e^{-ia} = q$.   
By the Russo-Dye theorem (see for example Exercise 10.5.4 of \cite{KR}) it follows that $q$ is central in $xx^*Axx^*$. 
In this case, we have 
$y''_{\theta} := yxx^* + (y\cos{\theta} + \sin{\theta})(1-xx^*) \in \E(A)$
for $\theta \in \mathbb{R}$, and simple calculations show that
\[
\frac{1}{\sqrt{2}}(x \pm y''_{\theta}) =
\frac{1}{\sqrt{2}}((x \pm y)xx^* + (x \pm (y\cos{\theta} + \sin{\theta}))(1-xx^*))
\]
are partial isometries. 
In particular, $y''_{\theta} \in \A_x$. Since $y''_{\theta} \to y$ $(\theta \to 0)$, $y$ is not isolated in $\A_x$.

If $q$ is not isolated in $\P(xx^*Axx^*)$, take $q_n \in \P(xx^*Axx^*)$, $n \in \mathbb{N} = \{1, 2, \ldots\}$ such that $q \neq q_n \to q \, (n \to \infty)$.
Then we have $y \neq i(xx^* - 2q_n)x =: y'''_n\in \A_x$ and $y'''_n \to y \, (n \to \infty)$.
\end{proof}

Now we can prove the main theorem of this section.
\begin{theorem}\label{thm-C*}
Let $A$ and $B$ be unital C$^*$-algebras and $T : S(A)\to S(B)$ be a surjective isometry.  Then $T(\U(A)) = \U(B)$.
\end{theorem}
\begin{proof}
We know by Proposition \ref{prop-T} that $T(\E(A)) = \E(B)$ and by $(b)$ of Proposition \ref{prop-face} that $T(-x) = -T(x)$ for every $x \in \E(A)$.
It follows that $T(\A_x) = \A_{T(x)}$ for every $x \in \E(A)$. Therefore the preceding lemma implies $T(\U(A)) = \U(B)$.
\end{proof}
Recall the following theorem due to Hatori and Moln\'ar \cite{HM}:
\begin{theorem}[Hatori and Moln\'ar {\cite[Theorem 1]{HM}}]\label{thm-HM}
Let $A$ and $B$ be unital C$^*$-algebras and $T : \U(A)\to \U(B)$ be a surjective isometry. 
Then there exists a real linear surjective isometry $\widetilde{T} : A \to B$ which satisfies $T(e^{ia}) = \widetilde{T}(e^{ia})$ for every $a \in A_{sa}$. 
In fact, there exists a Jordan $^*$-isomorphism $J : A \to B$ and a central projection $p \in B$ such that $\widetilde{T}(x) = T(1)(pJ(x) + (1-p)J(x)^*)$ for all $x \in A$.
\end{theorem}
The Russo-Dye theorem shows that such a linear isometry is unique.
In order to solve Tingley's problem between unital C$^*$-algebras, it suffices to show that $\Phi := \widetilde{T}^{-1} \circ T : S(A) \to S(A)$ (in the sense of the preceding theorems) is equal to the identity mapping on $S(A)$.
Before we end this section we give an additional partial result.
The notation $A^{-1}$ means the set of invertible elements for a unital C$^*$-algebra $A$.

\begin{proposition}\label{prop-C*}
Let $A$ be a unital C$^*$-algebra and $\Phi : S(A) \to S(A)$ be a surjective isometry such that $\Phi(e^{ia}) = e^{ia}$ for every $a\in A_{sa}$. Then $\Phi (x) = x$ for every $x \in S(A) \cap A^{-1}$.
\end{proposition}
\begin{proof}
First we show $\Phi(u) = u$ for an arbitrary unitary $u \in \U(A)$. 
Consider the functional calculus in $A^{**}$, and set $v_1 = u \chi_{\{\operatorname{Re} z \geq 0\}} (u), v_2 =u \chi_{\{\operatorname{Re} z \leq 0\}} (u) \in A^{**}$, which are partial isometries belonging locally to $A$.
Take continuous functions $f,\, g : \mathbb{T} \to \{z \in \mathbb{T} \mid \operatorname{Re} z \geq 0\}$ 
which satisfy the following two properties: 
$f(z) =z =g(z)$ for every $z \in \mathbb{T}$ with $\operatorname{Re} z \geq 0$, 
and $\operatorname{Im} f(z) > \operatorname{Im} g(z)$ for every $z \in \mathbb{T}$ with $\operatorname{Re} z < 0$. 
It follows that $v_1$ is the maximum partial isometry in the collection of every partial isometry $v_0\in A^{**}$ which satisfies $f(u)v_0^* =v_0v_0^*$ and $g(u)v_0^* =v_0v_0^*$. 
Thus the minimum norm-closed face in $\B_A$ which contains both $f(u)$ and $g(u)$ is 
the face $\{x \in S(A) \mid xv_1^* =v_1v_1^*\}$. 
Since $f(u), g(u) \in e^{iA_{sa}}$, it follows that $\Phi (\{x \in S(A) \mid xv_1^* =v_1v_1^*\}) = \{x \in S(A) \mid xv_1^* =v_1v_1^*\}$. 
Similarly, $\Phi (\{x \in S(A) \mid xv_2^* =v_2v_2^*\}) = \{x \in S(A) \mid xv_2^* =v_2v_2^*\}$.
Since $\{x \in S(A) \mid xv_1^* =v_1v_1^*\} \cap \{x \in S(A) \mid xv_2^* =v_2v_2^*\} = \{u\}$, we obtain $\Phi(u) =u$.\smallskip

Next we show $\Phi(a) = a$ for an arbitrary positive invertible element $a \in S(A) \cap A^{-1} \cap A_+$.
Set $c := \min(\sigma(a))\, (> 0)$ and 
\[
\mathcal{S} := \{u \in \U(A) \mid \|u - a\| = 1-c\} 
= \{u \in \U(A) \mid \|u - \Phi(a)\| = 1-c\}.
\]

We see $\operatorname{Re}\lambda \geq c/2$ for every $\lambda \in \sigma(u)$, $u \in \mathcal{S}$. 
Assume there exists a $\lambda \in \sigma(u)$ such that $\operatorname{Re} \lambda < c/2$. 
Realizing $A \subset \B(\H)$, we obtain unit vectors $\xi_n \in \H$, $n \in \mathbb{N}$ such that 
$\|u\xi_n - \lambda\xi_n\| \to 0$ $(n \to \infty)$. 
Then it follows that $\lim_{n \to \infty} \langle u\xi_n, \xi_n \rangle = \lambda$ and $\langle a\xi_n, \xi_n \rangle \geq c$ for every $n \in \mathbb{N}$.
We have $\|u - a\| \geq 1$, a contradiction.

We consider the surjective isometry $u \mapsto u^*$ on $\mathcal{S}$. 
By the observation above, it follows that 
\[
\|u^* - u\| = \|1 - u^2\| = \|(1+u)(1-u)\| \geq \left(1+ \frac{c}{2}\right) \|1-u\|
\]
for every $u \in \mathcal{S}$.
Since $1^* = 1 \in \mathcal{S}$, $1 + c/2 > 1$ and $\mathcal{S}$ is bounded, it follows by \cite[Theorem 1.2]{V} that 
$\tau(1) = 1$ for every surjective isometry $\tau: \mathcal{S}\to\mathcal{S}$.

Since $\|1 - \Phi(a)\| = \|1-a\| = 1-c < 1$, the polar decomposition $\Phi(a) = v|\Phi(a)|$ satisfies $v \in \U(A)$. 
For $u \in \mathcal{S}$, we have
\[
\|vu^*v - \Phi(a)\| = \|u^*v - |\Phi(a)|\| = \|v^*u - |\Phi(a)|\| = \|u - \Phi(a)\| = 1-c.
\]
Thus the mapping $u \mapsto vu^*v$ is a surjective isometry on $\mathcal{S}$. 
Therefore, by the commented result in \cite{V}, it follows that $1 = v1^*v = v^2$. 
Combining this with the equation $\|v^* - |\Phi(a)|\| = \|1-\Phi(a)\| = 1-c$, we obtain $v = 1$. i.e.\ $\Phi(a)$ is positive. 

Take the continuous function $f_0 : [c, 1] \to \{z \in \mathbb{T} \mid \operatorname{Im} z \geq 0\}$ 
which is uniquely determined by the condition $|t-f_0(t)| = 1 + c$, $t \in [c, 1]$. Put $w := f_0(a)$. Then $(a-w)/(1+c)$ is a unitary.
Assume $\Phi(a) \not\leq a$. Then there exist $\lambda > 0$ and 
unit vectors $\eta_n \in \H$, $n \in \mathbb{N}$ such that 
$\|(\Phi(a) - a)\eta_n - \lambda\eta_n\| \to 0$ $(n \to \infty)$. 
We have 
\[
\langle(\Phi(a)-w)\eta_n, (a-w)\eta_n\rangle 
= \langle(\Phi(a)-a)\eta_n, (a-w)\eta_n\rangle + (1+c)^2.
\]
We know that $\lim_{n \to \infty} \langle(\Phi(a)-a - \lambda)\eta_n, (a-w)\eta_n\rangle = 0$.
Since $\operatorname{Re} (t-f_0(t)) \geq \sqrt{c^2+2c}$ for every $t \in [c, 1]$, we also know that
\[
\operatorname{Re} \lambda\langle\eta_n, (a-w)\eta_n\rangle = \frac{\lambda}{2}\langle\eta_n, ((a-w) + (a-w)^*)\eta_n\rangle \geq \lambda\sqrt{c^2+2c} > 0
\]
for every $n \in \mathbb{N}$. 
We have 
\[
\begin{split}
(1+c)^2 = \|\Phi(a)-w\|\|a-w\| &\geq \varlimsup_{n \to \infty} \operatorname{Re} \langle(\Phi(a)-w)\eta_n, (a-w)\eta_n\rangle \\
&= \varlimsup_{n \to \infty} \operatorname{Re}\lambda\langle\eta_n, (a-w)\eta_n\rangle + (1+c)^2 > (1+c)^2,
\end{split}
\] 
a contradiction.
Therefore we obtain $\Phi(a) \leq a$ and similarly $a \leq \Phi(a)$.\smallskip

Lastly we show $\Phi(x) = x$ for an arbitrary $x \in S(A) \cap A^{-1}$. 
The polar decomposition $x = u_0|x|$ satisfies $u_0 \in \U(A)$ and $|x| \in S(A) \cap A^{-1} \cap A_+$. 
Consider the surjective isometry $\Psi : S(A) \to S(A)$ which is defined by $\Psi(y) := u_0^{-1}\Phi(u_0y)$, $y \in S(A)$. 
Then the first part of this proof shows $\Psi(u) = u$ for every $u \in \U(A)$. 
The second part of this proof shows $|x| = \Psi(|x|) = u_0^{-1}\Phi(u_0|x|)$, hence 
$\Phi(x) = \Phi(u_0|x|) = u_0|x| = x$.  
\end{proof}

\section{Tingley's problem between preduals of von Neumann algebras}\label{predual}
In this section, we present an affirmative answer to Tingley's problem when the two spaces are preduals of von Neumann algebras. 
Our theorem extends the result of Fern\'andez-Polo, Garc\'es, Peralta and Villanueva \cite{FGPV}, in which Tingley's problem for the spaces of trace class operators on complex Hilbert spaces is solved affirmatively.

Let $M$ be a von Neumann algebra. By $(f)$ of Theorem \ref{thm-AP2} we know that for every norm-closed proper face $\F \subset \B_{M_*}$ there exists a unique nonzero partial isometry $v \in M$ such that $\F = \{v\}_{\prime}$.\smallskip

Recall that for a metric space $(X, d)$ and nonempty subsets $X_1, X_2 \subset X$, the \emph{Hausdorff distance} between $X_1$ and $X_2$ is defined by 
\[d_H(X_1, X_2) := \max \{ \sup_{x \in X_1} \inf_{y \in X_2} d(x, y), \sup_{y \in X_2}\inf_{x \in X_1} d(x, y)\}.\]
Endow the space of nonzero partial isometries in $M$ with the distance $\delta_H(v, w) := d_H(\{v\}_{\prime}, \{w\}_{\prime})$. 
(It is easy to show that $\delta_H$ actually satisfies the axioms of distance.)

\begin{lemma}\label{lem-predual1}
Let $M \subset \B(\H)$ be a von Neumann algebra. 
\begin{enumerate}[$(a)$]
\item Let $w_1, w_2 \in M$ be nonzero partial isometries with $w_1^*w_1 = w_2^*w_2$ or $w_1w_1^* = w_2w_2^*$. Then $\|w_1 - w_2\| \geq \delta_H(w_1, w_2)$.
\item $\|u-v\| = \delta_H(u, v)$ for every $u \in \U(M)$ and every $v \in \E(M)$.  
\end{enumerate}
\end{lemma}
\begin{proof}
$(a)$ Suppose $w_1^*w_1 = w_2^*w_2$. Let $\varphi \in \{w_1\}_{\prime}$. 
Then defining $\psi (x) := \varphi (w_1w_2^*x)$ $(x \in M)$, 
we have $\psi \in \{ w_2\}_{\prime}$ and 
\[
\|\varphi - \psi\| = \|\varphi((w_1w_1^*  - w_1w_2^*) \, \cdot\, )\| \leq \|w_1w_1^* - w_1w_2^*\| \leq \|w_1 - w_2\|.
\] 
Therefore we obtain 
$\displaystyle \sup_{\varphi \in \{w_1\}_{\prime}} \inf_{\psi \in \{w_2\}_{\prime}} \|\varphi - \psi\| \leq \|w_1 - w_2\|$. 
Similarly we can see $\displaystyle \sup_{\psi \in \{w_2\}_{\prime}} \inf_{\varphi \in \{w_1\}_{\prime}} \|\varphi - \psi\| \leq \|w_1 - w_2\|$, and therefore $\delta_H(w_1, w_2) \leq \|w_1 - w_2\|$. 
A similar discussion can be applied in the case $w_1w_1^* = w_2w_2^*$, too.\smallskip

\noindent
$(b)$ Suppose first that $u, v \in \U(M)$. 
The inequality $\delta_H(u, v) \leq \|u - v\|$ follows from $(a)$.
We may assume $u = 1$. In that case, we have
$\|u-v\| = \|1-v\| = \sup_{\lambda \in \sigma(v)} |1-\lambda|$. Take $\lambda_0 \in \sigma(v)$ which attains this supremum.
Since $\lambda_0 \in \sigma(v)$, 
there exist $\xi_n \in \H$ with $\|\xi_n\| = 1$, $n \in \mathbb{N}$ such that
$\|v\xi_n - \lambda_0 \xi_n\| \to 0$ $(n \to \infty)$.
Define $\varphi_n := \omega_{\xi_n, v\xi_n} = \langle\,\cdot\ \xi_n, v\xi_n\rangle \, (\in \{v\}_{\prime})$. 
Then for every $\psi \in \{ 1 \}_{\prime}$ we have 
\[
\|\psi-\varphi_n\| \geq |\psi(1)-\varphi_n(1)| = |1 - \langle\xi_n, v\xi_n\rangle|
\to |1 - \overline{\lambda_0}| = \|1 - v\| \quad (n \to \infty). 
\]
Therefore we obtain $\delta_H(1, v) \geq \|1-v\|$.
The proof when $u, v \in \U(M)$ is completed.\smallskip

Let us assume next $v \notin \U(M)$. We may assume $u = 1$ and  $vv^* \neq 1$. In that case, it follows that $\|u-v\| = \|1-v\| = 2$.
Take a unit vector $\xi \in (1- vv^*)\H$.
Since $v^*v(1-vv^*) = 1-vv^*$, the system $\{v^n\xi\}_{n \in \mathbb{N}}$ is orthonormal.
Define $\eta_n := \sum_{k=1}^n (-1)^kv^k\xi$ and 
$\varphi'_n := n^{-1} \omega_{\eta_n, v\eta_n}$ $(\in \{v\}_{\prime})$ for $n \in \mathbb{N}$.
Then for every $\psi \in \{1\}_{\prime}$ we have
\[
\|\varphi'_n - \psi\| \geq |\varphi'_n(1) - \psi(1)| = \left| -\frac{n-1}{n} -1\right| = \frac{2n-1}{n} \to 2 \quad (n \to \infty).
\]
It follows that $\delta_H(1, v) \geq 2$. The inequality $\delta_H(1, v) \leq 2$ is trivial.
\end{proof}
Note that using the same discussion as in $(b)$, we also gain $\|w_1 - w_2\| = \delta_H(w_1, w_2)$ 
for every pair of nonzero partial isometries $w_1, w_2 \in M$ with $w_1^*w_1 = w_2^* w_2$ and $w_1w_1^* = w_2w_2^*$.

The author does not know whether $\|v-w\| = \delta_H(v, w)$ holds for every pair $v, w \in \E(M)$, but the following lemma which is similar to Lemma \ref{lem-C*} holds.

\begin{lemma}\label{lem-predual2}
Let $M$ be a von Neumann algebra and $x$ be in $\E(M)$. Then $x$ is in $\U(M)$ if and only if the set
$\widehat{\A}_x := \{ y \in \E(M) \mid  \delta_H(x, \pm y) \leq \sqrt{2}\}$
has an isolated point with respect to the metric $\delta_H$.
\end{lemma}
\begin{proof}
The proof is parallel to that of Lemma \ref{lem-C*}. 

Suppose $x$ is in $\U(M)$. The preceding lemma shows that $\delta_H(x, y) = \|x - y\|$ for every $y \in \E(M)$. 
By the same discussion as in the proof of Lemma \ref{lem-C*}, we obtain $\widehat{\A}_x = i(2\P(M) - 1)x \, (\subset \U(M))$, which have isolated points $\pm ix$.

Next suppose $x \notin \U(M)$ and $y \in \widehat{\A}_x$. 
We again use the argument as in the proof of the Lemma \ref{lem-C*}. Note that the operators $y'_c, y''_{\theta}, y'''_n$ have the same initial spaces as $y$. Hence it is not difficult to see that the preceding lemma shows $y$ is not isolated in $\widehat{\A}_x$ with respect to the metric $\delta_H$.
\end{proof}

We state the main theorem of this section:
\begin{theorem}\label{thm-predual}
Let $M$ and $N$ be von Neumann algebras and $T : S(M_*) \to S(N_*)$ be a surjective isometry.
Then there exists a unique real linear surjective isometry $\widetilde{T} : M_* \to N_*$ which extends $T$.
\end{theorem}
We start proving.

Since $T$ gives a bijection between the classes of maximal convex sets in unit spheres, a bijection $T_1 : \E(M) \to \E(N)$ is determined by
$T(\{v\}_{\prime}) = \{T_1(v)\}_{\prime}$, $v \in \E(M)$ (see Proposition \ref{prop-T}). 
We also have $T_1(\widehat{\A}_v) = \widehat{\A}_{T_1(v)}$ for any $v\in \E(M)$. 
By the preceding lemma and $(a)$ of Proposition \ref{prop-face}, 
$T_1$ restricts to a bijection between unitary groups. 
Moreover, by $(b)$ of Lemma \ref{lem-predual1}, this is a surjective isometry between unitary groups.   
By the theorem of Hatori and Moln\'ar, 
there exists a unique real linear surjective isometry $\widetilde{T_1} : M \to N$ such that 
$T_1(u) = \widetilde{T_1}(u)$ for all $u \in e^{M_{sa}} = \U(M)$. 
Note that $\widetilde{T_1}$ and $\widetilde{T_1}^{-1}$ are $\sigma$-weakly continuous since they can be expressed by Jordan $^*$-isomorphisms.

Now we can construct a real linear surjective isometry $T_2 : N^* \to M^*$ which is canonically determined by $\widetilde{T_1}$ as the following:
\[
(T_2\varphi)(x) = (\operatorname{Re} \varphi) (\widetilde{T_1} (x)) - i(\operatorname{Re} \varphi)(\widetilde{T_1} (ix)), \quad \varphi \in N^*, x \in M.
\]
By the $\sigma$-weak continuity of $\widetilde{T_1}$ and $\widetilde{T_1}^{-1}$, $T_2$ restricts to a real linear surjective isometry from $N_*$ onto $M_*$.
We would like to show that ${T_2}^{-1} : M_* \to N_*$ is the extension we wanted.
In order to show this, it suffices to show that the surjective isometry $\Phi := T_2 \circ T : S(M_*) \to S(M_*)$ is equal to the identity mapping on $S(M_*)$. 
We know that $\Phi(\{u\}_{\prime}) = \{u\}_{\prime}$ for every $u \in \U(M)$.\smallskip

Let $v \in M$ be a nonzero partial isometry which has a unitary extension $u$. 
Then $\{v\}_{\prime} = \{u\}_{\prime} \cap \{2v-u\}_{\prime}$, and since 
$u$ and $2v-u$ are unitaries we have $\Phi(\{v\}_{\prime}) = \{v\}_{\prime}$.

Let $v \in M$ be a nonzero partial isometry which does not admit a unitary extension. Then there exist nonzero sub-partial isometries $v_1, v_2 \in M$ of $v$ which have unitary extensions and satisfy $v = v_1 + v_2$. 
(Indeed, we can take $v_1$ and $v_2$ as follows. 
Decompose the projection $v^*v$ to the sum of a finite projection $p_1$ and a properly infinite projection $p_2$. 
Since $v$ does not admit a unitary extension, we have $p_2\neq 0$. 
Decompose $p_2$ into the sum of mutually Murray-von Neumann equivalent projections $p_{21}$ and $p_{22}$. 
Then $v_1= v(p_1+p_{21})$ and $v_2= vp_{22}$ satisfy the condition. 
See for example \cite[Chapter 6]{KR} for information about the comparison theory of projections.)
Since $\{v\}_{\prime}$ is the minimum norm-closed face in $\B_{M_*}$ which contains both $\{v_1\}_{\prime}$ and $\{v_2\}_{\prime}$, we obtain $\Phi(\{v\}_{\prime}) = \{v\}_{\prime}$.

Therefore, in order to show that $\Phi$ is an identity mapping, it suffices to show $\Phi(\varphi) = \varphi$ for every normal state $\varphi$ on $M$ (i.e.\ for every $\varphi \in \{1\}_{\prime} = S(M_*) \cap M_{*+}$). 
Restricting our attention to $((\operatorname{supp}\varphi )M(\operatorname{supp}\varphi))_{*}$ which can be identified canonically with a subspace of $M_*$, 
we may also assume $\varphi$ is faithful 
(i.e.\ $\varphi(a) \neq 0$ for an arbitrary nonzero positive element $a$ in $M$). 
We need some more preparations.

\begin{lemma}\label{lem-predual3}
Let $M$ be a von Neumann algebra, $\varphi \in M_{*}$ be a self-adjoint element and $p \in \P(M)$. 
Then 
\[
\|\varphi\| \geq \sqrt{(\varphi(p) - \varphi(p^{\perp}))^2 + 4\|\varphi(p\cdot p^{\perp})\|^2}.
\]
\end{lemma}
\begin{proof}
Take a partial isometry $v \in M$ such that $vv^* \leq p$, $v^*v \leq p^{\perp}$ and $\varphi(v) = \|\varphi(p\cdot p^{\perp})\|$.
Then for every $\theta \in \mathbb{R}$ we have 
$p\cos\theta - p^{\perp}\cos\theta + (v+v^*)\sin\theta \in \B_{M_{sa}}$.
Take the supremum of
\[
\varphi(p\cos\theta - p^{\perp}\cos\theta + (v+v^*)\sin\theta)
= (\varphi(p) - \varphi(p^{\perp}))\cos\theta + 2\|\varphi(p\cdot p^{\perp})\|\sin\theta
\]
with respect to $\theta \in \mathbb{R}$.
\end{proof}

\begin{lemma}\label{lem-predual4}
Let $M$ be a von Neumann algebra, $\varphi$ be a normal state on $M$ and $p$ be in $\P(M)$. 
Suppose $0 < \varphi(p) < 1$. Put $\lambda:=\varphi(p)$.
Then the following two conditions are equivalent:
\begin{enumerate}[$(a)$]
\item There exist $\psi_1 \in \{p\}_{\prime}$ and $\psi_2 \in \{p^{\perp}\}_{\prime}$ such that $\|\varphi - \psi_1\| = 2(1-\lambda)$ and $\|\varphi - \psi_2\| = 2\lambda$.
\item $\varphi(p\cdot p^{\perp}) = 0 = \varphi(p^{\perp}\cdot p)$. 
\end{enumerate}
\end{lemma}
\begin{proof}
$(b) \Rightarrow (a)$ Put $\psi_1 := \lambda^{-1}\varphi(p \cdot p)$ and $\psi_2 := (1-\lambda)^{-1} \varphi(p^{\perp} \cdot p^{\perp})$.\smallskip

\noindent
$(a) \Rightarrow (b)$ 
If $(b)$ is not true, then $\|\varphi(p\cdot p^{\perp})\| > 0$. 
Therefore, by the preceding lemma, we have
\[
\|\psi_1-\varphi\| \geq \sqrt{((1-\varphi(p)) + \varphi(p^{\perp}))^2 + 4\|\varphi(p\cdot p^{\perp})\|^2} > (1-\varphi(p)) + \varphi(p^{\perp}), 
\]
\[
\|\varphi-\psi_2\| \geq \sqrt{(\varphi(p) + (1-\varphi(p^{\perp})))^2 + 4\|\varphi(p\cdot p^{\perp})\|^2} > \varphi(p) + (1-\varphi(p^{\perp})) 
\]
for every $\psi_1 \in \{p\}_{\prime}$ and every $\psi_2 \in \{p^{\perp}\}_{\prime}$.
It follows that $\|\varphi-\psi_1\| + \|\varphi-\psi_2\| > 2$, so $(a)$ is not true.
\end{proof}

We return to the proof of Theorem \ref{thm-predual}. Our task is to show $\Phi(\varphi) = \varphi$ for every normal faithful state $\varphi$ on $M$. 
Set $\varphi_0 := \Phi(\varphi)\, (\in \Phi(\{1\}_{\prime}) = \{1\}_{\prime})$. Assume $\varphi \neq \varphi_0$. 
Consider the Jordan decomposition of $\varphi - \varphi_0\, (\neq 0)$. We obtain a nonzero projection $p \in \P(M)$ such that $\varphi(p) < \varphi_0(p)$ and
\[
\varphi(p\cdot p) \leq \varphi_0(p\cdot p),\,\, \varphi(p\cdot p^{\perp}) = \varphi_0(p\cdot p^{\perp}),\,\, 
\varphi(p^{\perp}\cdot p^{\perp}) \geq \varphi_0(p^{\perp}\cdot p^{\perp}).
\]
Put $\lambda := \varphi(p)$. Then $0 < \lambda < 1$. 
Set
\[
\mathcal{S}_0 := \{\psi \in \{1\}_{\prime}\mid \psi(p) = \lambda,\, \psi(p\cdot p^{\perp}) =0= \psi(p^{\perp}\cdot p)\}.
\]
By the preceding lemma, $\mathcal{S}_0$ is equal to 
\[
\{\psi \in \{1\}_{\prime}\mid \|\psi-\psi_1\| = 2(1-\lambda),\,\, \|\psi-\psi_2\| = 2\lambda\,\text{ for some }\, \psi_1\in \{p\}_{\prime},\,\,\psi_2\in \{p^{\perp}\}_{\prime} \}. 
\]
Thus the equations $\Phi(\{p\}_{\prime}) = \{p\}_{\prime}$, $\Phi(\{p^{\perp}\}_{\prime}) = \{p^{\perp}\}_{\prime}$ imply $\Phi(\mathcal{S}_0) = \mathcal{S}_0$. 
In particular, we have 
$\inf_{\psi\in\mathcal{S}_0} \|\varphi - \psi\| = \inf_{\psi\in\mathcal{S}_0} \|\varphi_0 - \psi\|$.
However, Lemma \ref{lem-predual3} implies
\[
\inf_{\psi\in\mathcal{S}_0} \|\varphi - \psi\| = \|\varphi - (\varphi(p\cdot p) + \varphi(p^{\perp}\cdot p^{\perp}))\| = 2\|\varphi(p\cdot p^{\perp})\|
\]
and
\[
\begin{split}
\inf_{\psi\in\mathcal{S}_0} \|\varphi_0 - \psi\| & \geq 
\inf_{\psi\in\mathcal{S}_0} \sqrt{((\varphi_0(p)-\psi(p)) - (\varphi_0(p^{\perp})-\psi(p^{\perp}))^2 + 4\|\varphi(p\cdot p^{\perp}))\|^2} \\
& = \sqrt{4(\varphi_0(p)-\varphi(p))^2 + 4\|\varphi(p\cdot p^{\perp})\|^2}.
\end{split}
\]
We have a contradiction.
The proof of Theorem \ref{thm-predual} is completed. $\hfill\Box$

\begin{corollary}\label{cor-predual}
Let $A$ and $B$ be C$^*$-algebras and $T: S(A^*) \to S(B^*)$ be a surjective isometry. (We do not assume $A$ or $B$ is unital.) 
Then there exists a unique real linear surjective isometry $\widetilde{T}: A^* \to B^*$ which extends $T$. 
\end{corollary}
\begin{proof}
We know that $A^*$ and $B^*$ can be considered as the preduals of the enveloping von Neumann algebras $A^{**}$ and $B^{**}$, respectively. 
Thus we can apply Theorem \ref{thm-predual}.
\end{proof}

\section{Tingley's problem between the spaces of self-adjoint elements}\label{sa}
To solve Tingley's problem between the spaces of self-adjoint elements in (preduals of) von Neumann algebras, 
it seems to be difficult to make use of the set of self-adjoint unitaries 
because the theorem of Hatori and Moln\'ar (Theorem \ref{thm-HM}) cannot be applied in this case. 
What we use in this section is the structure of projection lattices of von Neumann algebras, 
but note that in general a surjective isometry between projection lattices cannot be extended to a linear surjective isometry. 
For example, every bijection from $\P(\ell^{\infty})$ onto itself is automatically isometric.

However, combining the metric condition with a condition about orthogonality, we see that a mapping between projection lattices can be extended linearly. 
We rely on the following theorem due to Dye \cite{Dy}. 
Let $M, N$ be von Neumann algebras. 
A bijection $T : \P(M) \to \P(N)$ (or $\P(M)\setminus \{0\} \to \P(N)\setminus \{0\}$) is called an \emph{orthoisomorphism} 
if for any projections $p, q \in \P(M)$ (or $\P(M)\setminus\{0\}$),\, $pq = 0 \Longleftrightarrow T(p)T(q) = 0$.

\begin{theorem}[Dye {\cite[Corollary of Theorem 1]{Dy}}]\label{thm-D}
Let $M$ and $N$ be von Neumann algebras and $T : \P(M) \to \P(N)$ be an orthoisomorphism. 
Suppose $M$ does not have a type I$_2$ summand. 
Then there exists a unique linear surjective isometry $\widetilde{T} : M_{sa} \to N_{sa}$ which extends $\Phi$.
\end{theorem}
The condition $M$ does not have a I$_2$ summand is inevitable in general cases. In order to drop this condition, we add another condition.

\begin{proposition}\label{prop-sa1}
Let $M$ and $N$ be von Neumann algebras and $T : \P(M) \to \P(N)$ be an orthoisomorphism. 
Suppose $\|p-q\| = \|\Phi(p) - \Phi(q)\|$ for every pair of maximal abelian projections $p, q$ in the type I$_2$ summand of $M$. 
Then there exists a unique linear surjective isometry $\widetilde{T} : M_{sa} \to N_{sa}$ which extends $T$.
\end{proposition}
\begin{proof}
It suffices to show this proposition when $M$ and $N$ are of type I$_2$.
Since $T$ restricts to a bijection between the classes of central projections, it follows that $M$ is $^*$-isomorphic to $N$.
We decompose $M$ as $M = M_2(A)$ using an abelian von Neumann algebra $A$. Then the element
$T
\begin{pmatrix}
1 & 0 \\
0 & 0
\end{pmatrix}$
is a maximal abelian projection in $N \,(\cong M_2(A))$. 
Taking an appropriate $^*$-isomorphism from $N$ onto $M_2(A)$, we may assume 
$M = N = M_2(A)$, 
$T
\begin{pmatrix}
1 & 0 \\
0 & 0
\end{pmatrix}
=
\begin{pmatrix}
1 & 0 \\
0 & 0
\end{pmatrix}$
and every nonzero central projection in $M_2(A)$ is fixed under $T$.
We also have 
\[
T\begin{pmatrix}
0 & 0 \\
0 & 1
\end{pmatrix}
= T\left(\begin{pmatrix}
1 & 0 \\
0 & 0
\end{pmatrix}^{\perp}\right)
= T\begin{pmatrix}
1 & 0 \\
0 & 0
\end{pmatrix}^{\perp}
=\begin{pmatrix}
0 & 0 \\
0 & 1
\end{pmatrix}.
\]
Thus $T$ restricts to a bijection from 
\[
\left\{ p \in \P(M_2(A)) \middle| 
\left\| p - 
\begin{pmatrix}
1 & 0 \\
0 & 0
\end{pmatrix}
\right\|
= \frac{1}{\sqrt{2}} = \left\| p - 
\begin{pmatrix}
0 & 0 \\
0 & 1
\end{pmatrix}
\right\| \right\}
= 
\left\{
\frac{1}{2}
\begin{pmatrix}
1 & u \\
u^* & 1
\end{pmatrix}
\middle| u \in \U(A)
\right\}
\]
onto itself. There exist $u_1, u_i \in \U(A)$ such that
\[
T\left( \frac{1}{2}
\begin{pmatrix}
1 & 1 \\
1 & 1
\end{pmatrix}
\right) = \frac{1}{2}
\begin{pmatrix}
1 & u_1 \\
u_1^* & 1
\end{pmatrix}, \quad
T\left( \frac{1}{2}
\begin{pmatrix}
1 & i \\
-i & 1
\end{pmatrix}
\right) = \frac{1}{2}
\begin{pmatrix}
1 & u_i \\
u_i^* & 1
\end{pmatrix}.
\]
Since 
\[
\frac{\|u_1-u_i\|}{2} =
\left\|
\frac{1}{2}
\begin{pmatrix}
1 & u_1 \\
u_1^* & 1
\end{pmatrix}
- \frac{1}{2}
\begin{pmatrix}
1 & u_i \\
u_i^* & 1
\end{pmatrix}\right\|
= 
\left\|
\frac{1}{2}
\begin{pmatrix}
1 & 1 \\
1 & 1
\end{pmatrix}
- \frac{1}{2}
\begin{pmatrix}
1 & i \\
-i & 1
\end{pmatrix}\right\|=\frac{1}{\sqrt{2}}
\]
and
\[ 
\frac{\|u_1+u_i\|}{2} =
\left\|
\frac{1}{2}
\begin{pmatrix}
1 & u_1 \\
u_1^* & 1
\end{pmatrix}
- \frac{1}{2}
\begin{pmatrix}
1 & u_i \\
u_i^* & 1
\end{pmatrix}^{\perp}\right\|
= 
\left\|
\frac{1}{2}
\begin{pmatrix}
1 & 1 \\
1 & 1
\end{pmatrix}
- \frac{1}{2}
\begin{pmatrix}
1 & i \\
-i & 1
\end{pmatrix}^{\perp}\right\|=\frac{1}{\sqrt{2}},
\]
it follows that $ (u_1 \pm u_i)/\sqrt{2} \in \U(A)$.
We define a linear surjective isometry $\widetilde{T}: M_2(A)_{sa} \to M_2(A)_{sa}$ by 
\[
\widetilde{T}
\begin{pmatrix}
a_1 & a_2 + a_3i \\
a_2 - a_3i & a_4
\end{pmatrix} 
= 
\begin{pmatrix}
a_1 & a_2u_1 + a_3u_i \\
a_2u_1^* + a_3u_i^* & a_4
\end{pmatrix}, \quad a_1, a_2, a_3, a_4 \in A_{sa}.
\]
Let $c \in \mathbb{T}$. Consider the distance from $\displaystyle\frac{1}{2}
\begin{pmatrix}
1 & c\\
\overline{c} & 1
\end{pmatrix}$ to 
\[
\frac{1}{2}\begin{pmatrix}
1 & 1 \\
1 & 1
\end{pmatrix},\quad \frac{1}{2}\begin{pmatrix}
1 & -1 \\
-1 & 1
\end{pmatrix} \left(= \left(\frac{1}{2}\begin{pmatrix}
1 & 1 \\
1 & 1
\end{pmatrix}\right)^{\perp}\right),
\]
\[
\frac{1}{2}\begin{pmatrix}
1 & i \\
-i & 1
\end{pmatrix}\quad \text{and}\quad \frac{1}{2}\begin{pmatrix}
1 & -i \\
i & 1
\end{pmatrix}\left(= \left(\frac{1}{2}\begin{pmatrix}
1 & i \\
-i & 1
\end{pmatrix}\right)^{\perp}\right).
\]
Then we easily obtain 
\[
T\left( \frac{1}{2}
\begin{pmatrix}
1 & c\\
\overline{c} & 1
\end{pmatrix}
\right)  = \frac{1}{2}
\begin{pmatrix}
1 & u_1\operatorname{Re} c + u_i\operatorname{Im} c\\
u_1^*\operatorname{Re} c + u_i^*\operatorname{Im} c & 1
\end{pmatrix} = \widetilde{T} \left( \frac{1}{2}
\begin{pmatrix}
1 & c\\
\overline{c} & 1
\end{pmatrix}
\right).
\]
A similar consideration shows that
\[
\begin{split}
T\left( \frac{1}{1+|c|^2}
\begin{pmatrix}
1 & c\\
\overline{c} & |c|^2
\end{pmatrix}
\right) & = \frac{1}{1+|c|^2}
\begin{pmatrix}
1 & u_1\operatorname{Re} c + u_i\operatorname{Im} c\\
u_1^*\operatorname{Re} c + u_i^*\operatorname{Im} c & |c|^2
\end{pmatrix} \\
& = \widetilde{T} \left( \frac{1}{1+|c|^2}
\begin{pmatrix}
1 & c\\
\overline{c} & |c|^2
\end{pmatrix}
\right)
\end{split}
\]
for an arbitrary $c \in \mathbb{C}$.
Since $T$ is an orthoisomorphism, we obtain
\[
T\left( 
\sum_{n=1}^N \frac{1}{1+|c_n|^2}
\begin{pmatrix}
q_n & c_nq_n\\
\overline{c_n}q_n & |c_n|^2q_n
\end{pmatrix}
\right) 
 = \,  \widetilde{T}\left( 
\sum_{n=1}^N \frac{1}{1+|c_n|^2}
\begin{pmatrix}
q_n & c_nq_n\\
\overline{c_n}q_n & |c_n|^2q_n
\end{pmatrix}
\right)
\]
for arbitrary numbers $c_1, \ldots, c_N \in \mathbb{C}$ and projections $\{q_n\}_{n=1}^N \subset \P(A)$ with $q_1 + \cdots + q_N = 1 \in A$. 
The set
\[
\left\{
\sum_{n=1}^N \frac{1}{1+|c_n|^2}
\begin{pmatrix}
q_n & c_nq_n\\
\overline{c_n}q_n & |c_n|^2q_n
\end{pmatrix}
\middle| 
\begin{matrix}
N \in \mathbb{N},\,\,  c_1, \ldots, c_N \in \mathbb{C}, \\
q_1, \cdots, q_N \in \P(A),\,\, q_1 + \cdots + q_N = 1 \in A  
\end{matrix}
\right\}
\]
is norm-dense in the class of maximal abelian projections in $\P(M_2(A))$, so we have $T = \widetilde{T}$ on this class. 
Let $p\in M_2(A)$. Then $p$ can be decomposed as $p=
\begin{pmatrix}
q_0 & 0 \\
0 & q_0
\end{pmatrix}
+
\begin{pmatrix}
q_1 & 0 \\
0 & q_1
\end{pmatrix}
p_0$,
where $q_0$ and $q_1$ are mutually orthogonal projections in $A$ and $p_0$ is a maximal abelian projection in $M_2(A)$. 
Since $T$ is an orthoisomorphism, we obtain $T(p) = \widetilde{T}(p)$. 
\end{proof}

We also make use of the following proposition, whose proof can be found in the paper of Akemann and Pedersen \cite{AP}.
\begin{proposition}[See {\cite[Lemma 2.7]{AP}}]\label{prop-A}
Let $A$ be a C$^*$-algebra, $p$ be a compact projection, $q$ be an open projection with $p \leq q$. 
Then there exist a decreasing net $(x_{\alpha})$ and an increasing net $(y_{\alpha})$ in $A_{+}$ 
such that $p \leq x_{\alpha} , y_{\alpha} \leq q$ with the property 
$x_{\alpha}$ converges to $p$ and $y_{\alpha}$ converges to $q$ $\sigma$-strongly in $A^{**}$.
\end{proposition}
Using this, we obtain the following proposition.

\begin{proposition}\label{prop-sa2}
Let $M$ be a von Neumann algebra and $\F \subset \B_{M_{sa}}$ be a norm-closed proper face. 
Then $\F$ is $\sigma$-weakly closed if and only if there exists a unique element $x_{\F} \in \F$ such that $\|x_{\F} - y\| \leq 1$ for every $y \in \F$.
\end{proposition}
\begin{proof}
Suppose $\F$ is $\sigma$-weakly closed. Then by $(e)$ of Theorem \ref{thm-AP2} there exists a unique pair of projections
$p, q \in \P(M)$ such that $pq = 0$ and 
\[
\F = p - q +(1-p-q)\B_{M_{sa}}(1-p-q) = p - q +\B_{((1-p-q)M(1-p-q))_{sa}}.
\]
Then $x := p-q$ is the only element which satisfies given conditions.

Suppose $\F$ is not $\sigma$-weakly closed. 
By $(b)$ of Theorem \ref{thm-AP2}, there exists a unique pair of compact projections $p, q$ such that $pq=0$ and 
\[
\F = \{y \in M_{sa} \mid y(p-q) = p+q\} = \{y \in M_{sa} \mid 2p-1 \leq y \leq 1-2q\}.
\] 
Since $\F$ is not $\sigma$-weakly closed, at least one of $p$ and $q$ is not an element of $\P(M)$. 
Let $x$ be in $\F$. Then $0 \neq x-(p-q) \in (1-p-q)M^{**}_{sa}(1-p-q)$.
However, by the preceding lemma, there exist nets $(a_\alpha), (b_\alpha) \in \F$ 
such that $a_{\alpha} \searrow 2p-1$ and $b_{\alpha} \nearrow 1-2q$\, $\sigma$-strongly in $M^{**}$. 
Hence we have $x-a_{\alpha}\to x-2p+1 = (x-(p-q)) + (1-p-q)$ and 
$x-b_{\alpha}\to x-1+2q = (x-(p-q)) - (1-p-q)$ ($\sigma$-strongly), thus
$\varlimsup \|x - a_{\alpha}\| > 1$ or $\varlimsup \|x - b_{\alpha}\| > 1$. 
Therefore, there exists no element $x$ which satisfies the given conditions.
\end{proof}

Therefore, we can detect $\sigma$-weakly closed faces in $\B_{M_{sa}}$ from the class of norm-closed faces only by the metric structure of them.\smallskip

Recall that Mankiewicz's generalization of the Mazur-Ulam theorem states that 
every surjective isometry between open connected nonempty subsets of real Banach spaces 
extends to an affine surjective isometry between the whole spaces \cite{M}. 
Now we can prove the following theorem:

\begin{theorem}\label{thm-sa1}
Let $M, N$ be von Neumann algebras and $T : S(M_{sa}) \to S(N_{sa})$ be a surjective isometry. 
Then there exists a unique linear surjective isometry $\widetilde{T} : M_{sa} \to N_{sa}$ which extends $T$.
\end{theorem}
\begin{proof}
Propositions \ref{prop-T} and \ref{prop-sa2} imply that for a $\sigma$-weakly closed proper face $\F_1 \subset \B_{M_{sa}}$, 
$\F_2 := T(\F_1)$ is a $\sigma$-weakly closed proper face in $\B_{N_{sa}}$ and $T(x_{\F_1}) = x_{\F_2}$. 
Therefore $T$ restricts to a bijection between the classes of nonzero self-adjoint partial isometries. 
We also have $T(-v) = -T(v)$ for every nonzero self-adjoint partial isometry $v \in M$ by $(b)$ of Proposition \ref{prop-face}.

We know that $T$ restricts to a bijection from $\E(M_{sa})$ ($=\{2p-1\mid p\in\P(M)\}$, which is the collection of self-adjoint unitaries in $M$) onto $\E(N_{sa})$. 
Since $u \in \E(M_{sa})$ is central if and only if $u$ is isolated in $\E(M_{sa})$ (see the last two paragraphs of the proof of Lemma \ref{lem-C*}), 
it follows that $T(1)$ is central in $N$.

Define $T_1 : S(M_{sa}) \to S(N_{sa})$ by $T_1(x) := T(1)^{-1} T(x)$, $x \in S(M_{sa})$. 
What we have to do is to show that the mapping $T_1$ admits a linear extension.
We first show that $T_1$ restricts to an orthoisomorphism from $\P(M)\setminus \{0\}$ onto $\P(N)\setminus \{0\}$. 
We already know that $T_1$ restricts to an order-preserving bijection from $\P(M)\setminus \{0\}$ onto $\P(N)\setminus \{0\}$. 
Hence it suffices to show $T_1(p^{\perp}) = T_1(p)^{\perp}$ for an arbitrary projection $p \in \P(M)\setminus \{0,1\}$.
Let $p \in \P(M)\setminus \{0,1\}$.
Then $T_1$ restricts to a bijection from $p + p^{\perp}\B_{M_{sa}}p^{\perp}$ onto $T_1(p) + T_1(p)^{\perp}\B_{N_{sa}}T_1(p)^{\perp}$. 
Identify $p + p^{\perp}\B_{M_{sa}}p^{\perp}$ with 
$p^{\perp}\B_{M_{sa}}p^{\perp} = \B_{(p^{\perp}Mp^{\perp})_{sa}}$ and 
$T_1(p) + T_1(p)^{\perp}\B_{N_{sa}}T_1(p)^{\perp}$ with $\B_{(T(p)^{\perp}NT(p)^{\perp})_{sa}}$.
It follows by Mankiewicz's theorem that 
\[
T_1(p) = T_1\left(\frac{1}{2} ((p-p^{\perp}) + 1)\right)
= \frac{1}{2}(T_1(p-p^{\perp}) + T_1(1)) = \frac{1}{2}(T_1(p-p^{\perp}) + 1)
\]
Similarly we obtain $T_1(p^{\perp}) = (T_1(p^{\perp}-p) + 1)/2$. 
Since $T_1(p-p^{\perp})=-T_1(p^{\perp}-p)$, we obtain $T_1(p) + T_1(p^{\perp}) = 1$. 
i.e.\ $T_1(p^{\perp}) = T_1(p)^{\perp}$.

Thus Proposition \ref{prop-sa1} implies that there exists a linear surjective isometry $\widetilde{T_1} : M_{sa} \to N_{sa}$ 
such that $T_1(p) = \widetilde{T_1}(p)$ for every $p \in \P(M)\setminus \{0\}$.
Let $p \in \P(M)\setminus \{0\}$. 
Then both $T_1$ and $\widetilde{T_1}$ restrict to surjective isometries 
from $p + p^{\perp}\B_{M_{sa}}p^{\perp}$ onto $T(p) + T(p)^{\perp}\B_{N_{sa}}T(p)^{\perp}$. 
Moreover, they coincide on $\{p_0 \in \P(M)\setminus\{0\} \mid p \leq p_0 \}$, 
which is total in $\B_{(p^{\perp}Mp^{\perp})_{sa}}$ identified with $p + p^{\perp}\B_{M_{sa}}p^{\perp}$.
Hence Mankiewicz's theorem implies that 
$T_1(x) = \widetilde{T_1}(x)$ for every $x \in p + p^{\perp}\B_{M_{sa}}p^{\perp}$.
Similarly we have $T_1(x) = \widetilde{T_1}(x)$ for every $x \in -p + p^{\perp}\B_{M_{sa}}p^{\perp}$.
By the functional calculus, we know that the set 
\[
\bigcup_{p \in \P(M)\setminus \{0\}}\left(\left(p + p^{\perp}\B_{M_{sa}}p^{\perp} \right)\cup \left(-p + p^{\perp}\B_{M_{sa}}p^{\perp}\right)\right)
\]
is norm-dense in $S(M_{sa})$. Thus we obtain 
$T_1(x) = \widetilde{T_1}(x)$ for every $x \in S(M_{sa})$.
\end{proof}

\begin{remark}
In \cite[Theorem 5.8]{P}, using the Bunce-Wright-Mackey-Gleason theorem instead of Dye's theorem, 
Peralta gave another way to show the preceding theorem. 
\end{remark}

In order to think about the space of self-adjoint elements in the preduals of von Neumann algebras, we again use the Hausdorff distance as in Section \ref{predual}.

\begin{lemma}\label{lem-sa}
Let $M$ be a von Neumann algebra of type I$_2$. Then for arbitrary maximal abelian projections $p, q \in M$ we have $2 \|p-q\| = \delta_H(p, q)$.
\end{lemma}
\begin{proof}
As in the proof of Proposition \ref{prop-sa1}, we can decompose $M$ as $M = M_2(A)$ using an abelian von Neumann algebra $A$. We may assume $p = 
\begin{pmatrix}
1 & 0 \\
0 & 0
\end{pmatrix} \in \P(M_2(A))$. 

Let $q_1, q_2 \in\P(M_2(A))$. If $\|q_1-q_2\|$ is sufficiently small, 
for every $\varphi\in \{q_1\}_{\prime}$, 
it is not difficult to see that 
\[
\frac{\varphi(q_2\cdot q_2)}{\|\varphi(q_2\cdot q_2)\|} \in\{q_2\}_{\prime}\quad\text{and}\quad
\left\| \varphi-\frac{\varphi(q_2\cdot q_2)}{\|\varphi(q_2\cdot q_2)\|}\right\| \,\text{is small.}
\] 
Thus the mapping $\P(M_2(A)) \ni q \mapsto \delta_H(p, q)$ is continuous in the norm metric. 

Hence we may also assume that $q$ can be decomposed to the following form: 
there exist $N \in \mathbb{N}$, $q_1, \ldots, q_N \in \P(A)$ and $c_1, \ldots, c_N \in \mathbb{C}$ such that
\[
\sum_{n = 1}^{N} q_n = 1 \in A,\quad
q = \sum_{n = 1}^{N}
\frac{1}{1 + |c_n|^2}
\begin{pmatrix}
q_n & c_nq_n \\
\overline{c_n}q_n & |c_n|^2q_n
\end{pmatrix} \in M_2(A).
\]
In this case, we easily have 
\[
\|p - q\| = \max_{1 \leq n \leq N}
\left\|
\begin{pmatrix}
1 & 0 \\
0 & 0
\end{pmatrix} - 
\frac{1}{1 + |c_n|^2}
\begin{pmatrix}
1 & c_n \\
\overline{c_n} & |c_n|^2
\end{pmatrix}\right\|
= \max_{1 \leq n \leq N} \frac{|c_n|}{\sqrt{1 + |c_n|^2}}
\]
and
\[
\delta_H(p, q) = \max_{1 \leq n \leq N}
\delta_H\left(
\begin{pmatrix}
1 & 0 \\
0 & 0
\end{pmatrix}, 
\frac{1}{1 + |c_n|^2}
\begin{pmatrix}
1 & c_n \\
\overline{c_n} & |c_n|^2
\end{pmatrix} \right)
= \max_{1 \leq n \leq N} \frac{2|c_n|}{\sqrt{1 + |c_n|^2}}.
\] 
In particular, we obtain $2\|p-q\| = \delta_H(p,q)$.
\end{proof}

Let us recall the following well-known fact. 
\begin{lemma}
Let $M$ be a von Neumann algebra and $\varphi, \psi$ be normal states on $M$. 
Then we have $\|\varphi-\psi\| = 2$ if and only if $\operatorname{supp}\varphi\perp \operatorname{supp}\psi$. 
\end{lemma}
\begin{proof}
Suppose $\|\varphi-\psi\| = 2$. 
There exists a self-adjoint partial isometry $v\in M$ such that $\varphi(v)-\psi(v) = 2$. 
We decompose as $v= p-q$, where $p, q\in\P(M)$ are mutually orthogonal projections. 
Since $\varphi$ and $\psi$ are states, by the equation $\varphi(v)-\psi(v) = 2$, 
we have $\varphi(p) = 1$ and $\psi(q) = 1$, thus $\operatorname{supp}\varphi \leq p \perp q \geq \operatorname{supp}\psi$.  
The other implication is clear. 
\end{proof}

Now we are ready to prove the following theorem.

\begin{theorem}\label{thm-sa2}
Let $M, N$ be von Neumann algebras and $T : S(M_{*sa}) \to S(N_{*sa})$ be a surjective isometry. 
Then there exists a unique linear surjective isometry $\widetilde{T} : M_{*sa} \to N_{*sa}$ which extends $T$.
\end{theorem}
\begin{proof}
We know $T(\{1\}_{\prime})$ is a maximal convex set in $S(N_{*sa})$, so it can be written as $\{u\}_{\prime}$, where $u \in N$ is a self-adjoint unitary. 
Consider the space $\E(M_{sa})$ endowed with the metric $\delta_H$. 
By Lemma \ref{lem-predual1}, this metric is equal to the norm metric. 
Thus the same discussion as in the second paragraph in the proof of Theorem \ref{thm-sa1} shows that $u$ is central.

It suffices to show that the surjective isometry $T_1$ defined by $T_1(\varphi) := (T\varphi)(\,\cdot\, u)$, $\varphi \in S(M_{*sa})$ admits a linear isometric extension.
We can define $T_2 : \P(M)\setminus \{0\} \to \P(N)\setminus \{0\}$ by
$T_1(\{p\}_{\prime}) = \{T_2(p)\}_{\prime}$, $p \in \P(M)\setminus \{0\}$.

By the preceding lemma, it is easy to see that for $p, q \in \P(M)\setminus\{0\}$, $pq = 0 \Longleftrightarrow \operatorname{dist} (\{p\}_{\prime}, \{q\}_{\prime}) = 2$. 
It follows that $T_2$ is an orthoisomorphism. 
 
Since every orthoisomorphism restricts to a bijection between the classes of maximal abelian projections of the type I$_2$ summands, 
the preceding lemma implies that $\|p-q\| = \|T_2(p) - T_2(q)\|$ for arbitrary maximal abelian projections $p, q$ in the I$_2$ summand of $M$. 
Therefore by Proposition \ref{prop-sa1}, there exists a linear surjective isometry $\widetilde{T_2} : M_{sa} \to N_{sa}$ 
such that $T_2(p) = \widetilde{T_2}(p)$ for every $p\in \P(M)\setminus \{0\}$.
Then we can show that $(\widetilde{T_2}_*)^{-1}$ is the linear surjective isometry we wanted, using an argument similar to the proof of Theorem \ref{thm-predual}.  
\end{proof}

Like Corollary \ref{cor-predual}, we have the following corollary.
\begin{corollary}
Let $A$ and $B$ be C$^*$-algebras and $T: S(A^*_{sa}) \to S(B^*_{sa})$ be a surjective isometry. (We do not assume $A$ or $B$ is unital.) 
Then there exists a unique linear surjective isometry $\widetilde{T}: A^*_{sa} \to B^*_{sa}$ which extends $T$. 
\end{corollary}

A similar discussion can be applied to prove the next theorem. 
\begin{theorem}
Let $M$ and $N$ be von Neumann algebras. 
\begin{enumerate}[$(a)$]
\item Suppose $T: S(M_*) \cap M_{*+} \to S(N_*) \cap N_{*+}$ is a surjective isometry between normal state spaces. Then there exists a unique linear surjective isometry $\widetilde{T}: M_{*sa} \to N_{*sa}$ which extends $T$.
\item Suppose the dimension of $N$ is larger than one and $T: \B_{M_*} \cap M_{*+} \to \B_{N_*} \cap N_{*+}$ is a surjective isometry between normal quasi-state spaces. Then there exists a unique linear surjective isometry $\widetilde{T}: M_{*sa} \to N_{*sa}$ which extends $T$.
\end{enumerate}
\end{theorem}
\begin{proof}
$(a)$ For $p \in \P(M)\setminus\{0\}$, we easily see that 
\[
\{p\}_{\prime} = \{\varphi\in S(M_*) \cap M_{*+} \mid \|\varphi-\psi\| = 2 \,\,\text{for any}\,\, \psi \in \{p^{\perp}\}_{\prime}\}.
\] 
It follows that 
\[
\begin{split}
T(\{p\}_{\prime}) &= \{\varphi \in S(N_*) \cap N_{*+} \mid \|\varphi-T(\psi)\|=2\,\,\text{for any}\,\,\psi \in \{p^{\perp}\}_{\prime}\} \\
&= \{\varphi \in S(N_*)\cap N_{*+} \mid \operatorname{supp}\varphi \perp \operatorname{supp}T(\psi) \,\, \text{for any}\,\, \psi\in \{p^{\perp}\}_{\prime}\}. 
\end{split}
\] 
Thus there exists an orthoisomorphism $T_1: \P(M)\setminus\{0\} \to \P(N)\setminus\{0\}$ 
which satisfies $T(\{p\}_{\prime}) = \{T_1(p)\}_{\prime}$ for every $p \in \P(M)\setminus\{0\}$. 
Then Lemma \ref{lem-sa} and Proposition \ref{prop-sa1} show that $T_1$ admits a unique linear surjective isometric extension $\widetilde{T_1}: M_{sa} \to N_{sa}$.
Use discussions in Section \ref{predual} to show that this linear mapping is what we wanted.\smallskip

\noindent
$(b)$ First we see $T(0) = 0$. By \cite[Lemma 3.6]{LNW}, we have $T(0) = 0$ unless $N$ is equal to $\mathbb{C}\oplus\mathbb{C}$ or $M_2(\mathbb{C})$. 
It is easy to show $T(0) = 0$ if $N=\mathbb{C}\oplus\mathbb{C}$ or $N=M_2(\mathbb{C})$. 
Thus $T$ restricts to a bijection from $S(M_*) \cap M_{*+}$ onto $S(N_*) \cap N_{*+}$. 
Hence $(a)$ shows that there exists a unique linear surjective isometry $\widetilde{T}: M_{sa} \to N_{sa}$ 
such that $\widetilde{T}(\varphi)=T(\varphi)$ for all $\varphi \in S(M_*) \cap M_{*+}$.
It suffices to show that $\Phi:=\widetilde{T}^{-1} \circ T : \B_{M_*} \cap M_{*+} \to \B_{M_*} \cap M_{*+}$ is equal to the identity mapping. 
Let $\varphi \in B_{M_*} \cap M_{*+}$. 
Since $\Phi$ is a surjective isometry and $\Phi(\psi) = \psi$ for every $\psi\in  S(M_*) \cap M_{*+}$, the set
\[
\begin{split}
\mathcal{S}_{\varphi}&=\{\psi \in S(M_*) \cap M_{*+} \mid \varphi\leq\psi\}\\
&=\{\psi \in S(M_*) \cap M_{*+} \mid \|\psi-\varphi\|= 1-\|\varphi\|\}
\end{split}
\]
is equal to
\[
\begin{split}
\mathcal{S}_{\Phi(\varphi)}&=\{\psi \in S(M_*) \cap M_{*+} \mid \Phi(\varphi)\leq\psi\}\\
&=\{\psi \in S(M_*) \cap M_{*+} \mid \|\psi-\Phi(\varphi)\| = 1- \|\Phi(\varphi)\|\}.
\end{split}
\]
Since $\{\varphi_0 \in \B_{M_*} \cap M_{*+} \mid \|\varphi_0\|=\|\varphi\|\,\,\text{and}\,\, \varphi_0\leq\psi\,\,\text{for any}\,\,\psi\in\mathcal{S}_{\varphi}\} = \{\varphi\}$, 
we obtain $\Phi(\varphi)=\varphi$.
\end{proof}

Note that $(b)$ answers the question in \cite[Remark 3.11]{LNW} positively in the case $p=1$. 
See also \cite[Theorem (4.5)]{Ka65}, in which Kadison proved $(a)$ with an additional assumption of affinity, 
and \cite[Theorem 4]{MT}, in which the case $M=N=\B(\H)$ for $(a)$ is solved.

\section{Problems}\label{problem}
The results in this paper may be extended to Tingley's problem between various types of Banach spaces concerning operator algebras. 
In this section, we give some problems which seem to have new perspectives for the study of Tingley's problem in the setting of operator algebras.\smallskip

In Section \ref{C*}, we showed that Tingley's problem between unital C$^*$-algebras has a positive answer if and only if the following problem has a positive answer.
\begin{problem}\label{problem-C*}
Let $A$ be a unital C$^*$-algebra and $\Phi : S(A) \to S(A)$ be a surjective isometry. 
Suppose that $\Phi(x) = x$ for every $x \in S(A) \cap A^{-1}$.
Is $\Phi$ equal to the identity mapping on $S(A)$?
\end{problem}

In Section \ref{sa}, we used the projection lattice of a von Neumann algebra. 
Such a discussion is generally impossible in the cases of general unital C$^*$-algebras. 
So we propose the following problem.
\begin{problem}[Tingley's problem for the spaces of self-adjoint operators in unital C$^*$-algebras]
Let $A, B$ be unital C$^*$-algebras and $T : S(A_{sa}) \to S(B_{sa})$ be a surjective isometry. 
Does $T$ admit an extension to a linear surjective isometry $\widetilde{T} : A_{sa} \to B_{sa}$?
\end{problem}

Since $T(1)$ is isolated in $\E(B_{sa})$, we obtain that $T(1)$ is a central self-adjoint unitary in $B$.
By the theorem due to Kadison in the paper in 1952 \cite[Theorem 2]{Ka52},
 if such a $\widetilde{T}$ exists, then the mapping $T(1)^{-1} \widetilde{T}(\cdot)$ is the restriction of a Jordan $^*$-isomorphism from $A$ onto $B$.\smallskip

As another direction, we present the problem below.
\begin{problem}[Tingley's problem for noncommutative $L^p$-spaces]\label{problem-Lp}
Let $1 < p < \infty$, $p \neq 2$, $M, N$ be von Neumann algebras and $T : S(L^p(M)) \to S(L^p(N))$ be a surjective isometry between the unit spheres of (Haagerup) noncommutative $L^p$-spaces (with respect to fixed normal semifinite faithful weights). 
Does $T$ admit an extension to a real linear surjective isometry $\widetilde{T} : L^p(M) \to L^p(N)$?
\end{problem}
See \cite{PX} for information about noncommutative $L^p$-spaces. 
We mention that the noncommutative $L^p$-space $L^p(M)$ can be considered as the complex interpolation space between $L^1(M) = M_*$ and $L^{\infty}(M) = M$. 

Noncommutative $L^p$-spaces are strictly convex, so it is completely impossible to apply the facial method which is wholly used in this paper.
However, as the first step to challenge this problem, one may make use of the following property.
\begin{proposition}[Equality condition of the noncommutative Clarkson inequality, see {\cite[Theorem 2.3]{S}} for references]
Let $1 \leq p < \infty$, $p \neq 2$, $M, N$ be von Neumann algebras and $\xi, \eta \in L^p(M)$. Then 
\[
\|\xi + \eta\|^p + \|\xi - \eta\|^p = 2(\|\xi\|^p + \|\eta\|^p) 
\Longleftrightarrow \xi\eta^* = 0 = \xi^*\eta.
\]
\end{proposition} 
Therefore, in the setting of Problem \ref{problem-Lp}, for $\xi, \eta \in S(L^p(M))$, $\xi\eta^* = 0 = \xi^*\eta$ if and only if $T(\xi)T(\eta)^* = 0 = T(\xi)^*T(\eta)$.
See \cite{S} for a result about complex linear surjective isometry between noncommutative $L^p$-spaces.\bigskip

\textbf{Note added in proof} \quad After the submission of this paper, the author and Ozawa announced several results on Tingley’s problem in \cite{MO} which contain an affirmative solution of Problem \ref{problem-C*}.\bigskip

\textbf{Acknowledgements} \quad This paper is written for master's thesis of the author. 
The author appreciates Yasuyuki Kawahigashi who is the advisor of the author. 
This work was supported by Leading Graduate Course for Frontiers of Mathematical Sciences and Physics, MEXT, Japan.

\end{document}